\documentclass[a4paper,12pt]{article}
\usepackage[ansinew]{inputenc}
\usepackage{amsfonts}
\usepackage{latexsym}
\usepackage{amsmath}
\usepackage{amssymb}
\usepackage{eepic}
\usepackage{graphicx}
\usepackage{eepic}

\newcommand{\R}{\mathbb{R}}

\newcommand{\C}{\mathbb{C}}
\newcommand{\N}{\mathbb{N}}

\newcommand{\co}{\operatorname{co}}

\newtheorem{defin}{Definition}[section]
\newenvironment{definition}{\begin{defin}\rm}{\end{defin}}
\newtheorem{theorem}[defin]{Theorem}
\newtheorem{proposition}[defin]{Proposition}
\newtheorem{exa}[defin]{Example}
\newenvironment{example}{\begin{exa}\rm}{\end{exa}}
\newtheorem{lemma}[defin]{Lemma}

\newenvironment{proof}
{\noindent{\it Proof.}}{\hfill $\Box$\par\vspace{2.5mm}}
\newenvironment{remark}
{\par\vspace{2.5mm}\noindent{\bf Remark.}}{\par\vspace{2.5mm}}

\makeatletter
\renewcommand{\ps@myheadings}{%
\renewcommand{\@evenhead}%
{{\rm\thepage}\hfil{\sc Gundersen, Heittokangas, Ishizaki, Tohge, Wen}\hfil}%
\renewcommand{\@oddhead}%
{\hfil{{\sc Dual exponential polynomials and a problem of Ozawa}\hfil{\rm\thepage}}}%
\renewcommand{\@evenfoot}{}%
\renewcommand{\@oddfoot}{}%
}\makeatother 

\numberwithin{equation}{section}

\title{\bf\Large Dual exponential polynomials and\\ a problem of Ozawa}
\author{J.~Heittokangas, K.~Ishizaki, K.~Tohge and Z.-T.~Wen}
\date{}
\begin{document}
\maketitle

\begin{abstract}
Complex linear differential equations with entire coefficients are studied in the situation where
one of the coefficients is an exponential polynomial and dominates the growth of all the other
coefficients. If such an equation has an exponential polynomial solution $f$, then the order of $f$ 
and of the dominant coefficient are equal, and the two functions possess a certain duality property.
The results presented in this paper improve earlier results by some of the present authors, 
and the paper adjoins with two open problems.

\medskip
\noindent
\textbf{Key words:}
Dual exponential polynomials,
exponential sum,  finite order, linear differential equation, Ozawa's problem, value distribution.

\medskip
\noindent
\textbf{2020 MSC:} Primary 30D15; Secondary 30D35.
\end{abstract}

\thispagestyle{empty}

\section{Introduction}

Frei \cite{Frei} has proved that the differential equation
	\begin{equation}\label{Frei-eqn}
	f''+e^{-z}f'+\alpha f=0,\quad \alpha\in\C\setminus\{0\},
	\end{equation}
has a subnormal (that is, non-trivial and finite-order) solution if and only if $\alpha=-m^2$ for a positive
integer $m$. The subnormal solution $f$ is a polynomial in $e^z$ of degree $m$, that is, an exponential sum of the form
	\begin{equation}\label{Frei-f}
	f(z)=1+C_1e^z+\cdots+C_me^{mz}, \quad C_j\in \mathbb{C}.
	\end{equation}
It was discovered in \cite[Lemma~1]{WGH} that in this representation one has $C_j\neq 0$ for $1\leq j\leq m$. Substituting the subnormal solution $f$ into \eqref{Frei-eqn}, we get
	$$
	\sum_{j=1}^mC_j j^2e^{jz}+\sum_{j=1}^mC_j je^{(j-1)z}
	-m^2\sum_{j=1}^mC_j e^{jz}=m^2.
	$$
By the Borel-Nevanlinna theorem \cite[pp.~70, 108]{Gross}, or simply by an elementary observation on three polynomials in $e^z$, this gives rise to the recursive formula
	$$
	C_1=m^2, \quad (m^2-j^2)C_j=(j+1)C_{j+1}, \quad 1\leq j\leq m,
	$$
from which $C_j=\frac{1}{j!}\prod_{k=0}^{j-1}(m^2-k^2)$ for $1\leq j \leq m$.
Due to the presence of the transcendental coefficient $e^{-z}$, any solution of \eqref{Frei-eqn}
linearly independent with $f$ in \eqref{Frei-f} must be of infinite order \cite{GSW}.
For example, when $\alpha=-1$, the function $g(z)=\exp(e^{-z}+z)$ is an infinite order solution of \eqref{Frei-eqn} and linearly independent with $f(z)=1+e^z$.

Ozawa \cite{Ozawa} showed that if $a\neq 0$, then the non-trivial solutions of
	$$
	f''+e^{-z}f'+(az+b)f=0
	$$
are of infinite order of growth. If $P(z)$ is a non-constant polynomial, the question whether all
non-trivial solutions of
	$$
	f''+e^{-z}f'+P(z)f=0
	$$
are of infinite order of growth has been known as the Ozawa problem. This problem has been
answered affirmatively for particular polynomials $P(z)$ by Amemiya-Ozawa \cite{AO} and
by Gundersen \cite{Gary}, while the complete solution is by Langley \cite{Langley}.

We proceed to state three new examples of Frei-Ozawa type.

\begin{example}\label{Ozawa-ex}
If $H$ is an arbitrary entire function, then $f(z)=e^z+1$ solves
	$$
	f''+(H-1+He^{-z})f'-Hf=0.
	$$
Of particular interest is the case when $H$ is a polynomial.
\end{example}

\begin{example}
The function $f(z)=1+(1-3c)\left(e^z+\left(1-\frac34 c\right)\right)e^{2z}$, $c\in\C\setminus\{\frac13,\frac43\}$, with two exponential terms solves
	$$
	f''+\left(-\frac{5}{3}-c+\frac{2}{3}e^{-z}\right)f'+\left(2c-\frac{2}{3}\right)f=0.
	$$
\end{example}

\begin{example}\label{Kazuya-ex}
The function $f(z)=1+3e^{2z}+\sqrt{6}ie^{3z}$ with two exponential terms solves
	$$
	f''+\left(1-\sqrt{6}ie^{-z}+2e^{-2z}\right)f'-12f=0,
	$$
where the transcendental coefficient has two exponential terms also.
By making a change of variable $z\to wz$, where $w\in\C\setminus\{0\}$, we see that
	$$
	g''+w\left(1-\sqrt{6}ie^{-wz}+2e^{-2wz}\right)g'-12w^2g=0
	$$
has a solution $g(z)=1+3e^{2wz}+\sqrt{6}ie^{3wz}=f(wz)$.
\end{example}

One might wonder about possible examples of solutions $f$ with a single exponential term and of  transcendental coefficients $A(z)$ having at least two exponential terms. The non-existence 
of such examples will be confirmed in Theorem~\ref{one-term} below. For example, it will be shown 
that a function $f(z)=1+be^{w z}$ for $b, w \in \mathbb{C}$ is a solution of 
	$$ 
	f''+\left\{P_1(z)+P_2(z)e^{-wz}\right\}f'-P(z)f=0 
	$$ 
for $P(z), P_1(z), P_2(z)\in \mathbb{C}[z]$ if and only if $P_1(z)=\frac{1}{w}P(z)$ and $P_2(z)=\frac{1}{bw}P(z)$. 

In contrast to Ozawa's problem and complementing the three examples above, our primary focus is on exponential 
polynomial solutions of linear differential equations, in particular of second order  equations
	\begin{equation}\label{ldeAB}
	f''+A(z)f'+B(z)f=0,
	\end{equation}
where $A(z)$ and $B(z)$ are entire. An \emph{exponential polynomial} is a function of the form
    \begin{equation}\label{exp.eq}
    f(z)= P_1(z)e^{Q_1(z)} + \cdots + P_k(z)e^{Q_k(z)},
    \end{equation}
where $P_j$, $Q_j$ are polynomials for $1 \leq j \leq k$.
Observe that a polynomial is a special case of an exponential polynomial.
A transcendental exponential polynomial $f$ can be written in the \emph{normalized form}
    \begin{equation}\label{normalized-f}
    f(z)=F_0(z)+F_1(z)e^{w_1z^q}+\cdots+F_m(z)e^{w_mz^q},
    \end{equation}
where $q=\max\{\deg (Q_j)\}\geq 1$ is the order of $f$, the frequencies $w_j$ are non-zero and pairwise distinct, 
the multipliers $F_j$ are exponential polynomials of
order $\leq q-1$ such that $F_j(z)\not\equiv 0$ for $1\leq j\leq m$, and $m\leq k$ \cite{division, Stein1}.

\begin{definition}
(\cite{WGH}) Let $f$ be given in the normalized form \eqref{normalized-f}. If the non-zero
frequencies ${w}_1,\ldots,{w}_m$ of $f$ all lie on a fixed
ray $\arg(w)=\theta$, then $f$ is called a \emph{simple exponential polynomial}.
If $g$ is another simple exponential polynomial of the same order $q$ as $f$ such that
the non-zero frequencies of $g$ all lie on the opposite ray $\arg(w)=\theta+\pi$, then $f$ and $g$ are called \emph{dual exponential polynomials}.
\end{definition}

For example, the functions $f(z)=z^2e^{-iz}+ze^{z^2}+e^{2z^2+(1-i)z}$ and $g(z)=2e^{-z^2+(1+i)z}+z^2e^{-4z^2+iz}$ are dual exponential polynomials of order 2. 

In studying the differential equation \eqref{ldeAB} with entire coefficients $A(z)$ and $B(z)$, it is 
fundamental that each of its solutions $f$ is an entire function also. In this paper we study cases
when $f$ can be an exponential polynomial assuming that $A(z)$ is an exponential polynomial and that $B(z)$ grows
slowly compared to $A(z)$. Naturally, the set ${\mathcal E}$ of entire functions is a ring closed under differentiation, and the set $\operatorname{Exp}_q$ of exponential polynomials of order $\leq q\in\mathbb{N}$ together with constants in $\C$ and ordinary polynomials in $\mathbb{C}[z]=:\operatorname{Exp}_0$ becomes a differential subring of $\mathcal{E}$. On the other hand, $\operatorname{Exp}_q$ is not closed under integration in general except for the set $\operatorname{Exp}_1$ of exponential polynomials of order~$\leq 1$, which plays a role in our discussions. 

To identify a primitive of each element in $\operatorname{Exp}_1$, it is convenient to use the formula 
	$$ 
	\int z^n e^{w z}\, dz =\left(\frac{1}{w }z^n + 
	\sum_{\nu=0}^{n-1}\frac{(-1)^{n-\nu}n!}{w^{n-\nu+1}\nu!}z^\nu\right) e^{w z}+\text{constant}
	$$
for $n\in\mathbb{N}\cup\{0\}$ and $w\in\mathbb{C}\setminus\{0\}$. 
Of course, an analogous formula is not in general available for $z^n e^{w z^q}$ when $q\geq 2$. 
Indeed, recall the {\it error function} $\mathrm{erf}(z)$ defined also for complex argument $z$ by
	$$
	\mathrm{erf}(z) = \frac{2}{\sqrt{\pi}} \int_0^{z} e^{-\zeta^2}\,d\zeta. 
	$$ 
It is the primitive of $\frac{2}{\sqrt{\pi}}e^{-z^2}\in \operatorname{Exp}_2$, but the function itself 
is not an exponential polynomial. For this reason one needs the special expression $\mathrm{erf}(z)$ for this function as in the real argument case. This is also the case when $q\geq 3$. The value distribution
of the functions $\int_0^z e^{-\zeta^q}d\zeta$, $q\in\mathbb{N}$, as described in Nevanlinna's 
monograph \cite[pp.~168--170]{Nevanlinna}, is quite different from that of exponential polynomials
\cite{division, WH, Stein1}.

Along with $\operatorname{Exp}_{q-1}$, the set $\mathcal{S}_q(\theta)$ of simple exponential 
polynomials of order~$q$ with respect to a fixed angle $\theta\in[0,2\pi)$ forms a differential subring of $\operatorname{Exp}_q$. A unit element in $\mathcal{S}_q(\theta)$ is a single exponential term 
$e^{wz^q+p(z)}$ with $\arg(w)=\theta$, $p(z)\in \mathbb{C}[z]$ and $\deg(p)\leq q-1$, whose multiplicative inverse belongs to the set $\mathcal{S}_q(\theta+\pi)$ as its dual exponential polynomial. 
It should be observed that if $f\in \mathcal{S}_q(\theta)$ and $g\in \mathcal{S}_q(\theta+\pi)$ are
dual exponential polynomials, then $fg\in \operatorname{Exp}_{q-1}$ might not hold, but even so, the
growth of $fg$ in terms of the characteristic function could be somewhat reduced from that of $f$ or $g$.

\begin{example}
If $f(z)= e^z + e^{2z}$ and $g(z) = e^{-4z}$, then $T(r,f)=\frac{2}{\pi}r+O(\log r)$ and 
$T(r,g)=\frac{4}{\pi}r+O(\log r)$, while $T(r,fg)=\frac{3}{\pi}r+O(\log r)$, see \cite{WH}. Alternatively, the choice 
$g(z)=e^{-z}$ gives $T(r,g)=\frac{1}{\pi}r$ and $T(r,fg)=\frac{1}{\pi}r+O(1)$.
\end{example}

In our setting, the duality of two exponential polynomials is an interdependence among them in order to 
reduce the growth under multiplication, especially when combined with differentiation. For example, if 
$A$ and $f$ are dual exponential polynomials 
of order $q$, then $A$ and $f'$ are also dual, and at times the growth of $Af'$ is
reduced to $\rho(Af')<q$. 

The motivation for studying exponential polynomial solutions of \eqref{ldeAB} arises from the following previous result.

\begin{theorem}\textnormal{(\cite{WGH})}\label{second.theorem}
Suppose that $f$ is a transcendental exponential polynomial solution of \eqref{ldeAB},	
where $A(z)$ and $B(z)$ are exponential polynomials satisfying $\rho(B)<\rho(A)$. Then the following assertions hold.
\begin{itemize}
\item[\textnormal{(a)}] 
$f$ and $A(z)$ are dual exponential polynomials of order $q\in\N$, and $f$ has the 
normalized  representation
	\begin{equation}\label{simple-f}
	f(z)=c+F_1(z)e^{w_1z^q}+\cdots +F_m(z)e^{w_mz^q},
	\end{equation}
where $m\in\N$ and $c\in\C\setminus\{0\}$.
\item[\textnormal{(b)}]
If $\rho(Af')<q$, then $q=1$ and
    \begin{equation}\label{q=1}
   A(z)=a e^{-wz}, \quad B(z)=-w^2 \quad\text{and} \quad f(z)=c\left(1+\frac{w}{a}e^{wz}\right),
    \end{equation}
where $w=w_1$ and $a \in\C\setminus\{0\}$.
\end{itemize}
\end{theorem}

If $a=c=w=1$, then \eqref{q=1} reduces to Frei's equation \eqref{Frei-eqn} and Frei's solution
\eqref{Frei-f} in the case $m=1$. 
The following example illustrates that it is not always the case that the differential equation \eqref{ldeAB} 
possesses a non-trivial exponential polynomial solution when $A(z)$ and $B(z)$ are exponential polynomials 
satisfying $\rho(B)<\rho(A)$. 

\begin{example}\label{infiniteorder}
For a fixed $n\in\mathbb{Z}$, let $A(z)=-\frac{5}{3}+n+\frac{2}{3}e^{-z}$ and $B(z)=-\frac{8}{3}+n$.
Then \eqref{ldeAB} has a zero-free solution 
	\begin{equation}\label{f1}
	f(z)=\exp\left\{\dfrac{2}{3}e^{-z}+\left(\dfrac{8}{3}-n\right)z\right\}. 
	\end{equation}
Note that $f$ is an exponential of an exponential polynomial. Another solution of \eqref{ldeAB},
linearly independent with $f$, is
	\begin{eqnarray*}
	g(z)&=& f(z)\int^z \frac{e^{-\zeta}}{f(\zeta)}d\zeta \\
	&=& \exp\left\{\frac{2}{3}e^{-z}+\left(\frac{8}{3}-n\right)z\right\} 
	\int^z \exp\left\{\frac{2}{3}e^{-\zeta}+\left(\frac{5}{3}-n\right)\zeta\right\} d\zeta, 
	\end{eqnarray*}
where the integral represents an arbitrary primitive function. We may re-write this as
	$$
	g(z)\exp\left\{-\frac{2}{3}e^{-z}-\left(\frac{8}{3}-n\right)z\right\}
	=\int^z \exp\left\{\frac{2}{3}e^{-\zeta}+\left(\frac{5}{3}-n\right)\zeta	\right\}d\zeta
	$$
to see that $g$ solves a first order equation 
	$$
	g'(z)+\left\{-\frac{2}{3}e^{-z}+\left(\frac{8}{3}-n\right)\right\}g(z)
	=\exp\left\{\frac{4}{3}e^{-z}+\left(\frac{13}{3}-2n\right)z\right\}.  
	$$
This shows that $g$ cannot be any exponential polynomial as a function of infinite order. Hence it is necessary in Theorem~\ref{second.theorem} to assume 
that \eqref{ldeAB} has a nontrivial exceptional polynomial solution $f$. 
\end{example}

One may also observe that a small perturbation in the above coefficients $A(z)$ and $B(z)$ brings our desired case. In fact, by choosing $A(z)=-\frac{5}{3}-n+\frac{2}{3}e^{-z}$ and $B(z)=-\frac{2}{3}+2n$ for any $n\in\mathbb{Z}$, the 
equation \eqref{ldeAB} permits the  exponential polynomial solution 
	\begin{equation}\label{f2}
	f(z)=1+(1-3n)e^z+(1-3n)\left(1-\dfrac{3}{4}n\right)e^{2z}.
	\end{equation}
 A difference between these two cases can also be observed in the logarithmic derivatives: 
If $f$ is the function in \eqref{f1}, then $\frac{f'(z)}{f(z)}=-\frac{2}{3}e^{-z}+\frac{8}{3}-n$, while if $f$
is the function in \eqref{f2}, then $\frac{f'(z)}{f(z)}$ is not an exponential polynomial but an irreducible 
rational  function in $e^z$.

After discussing some properties of exponential polynomials in
Section~\ref{preliminaries.sec}, we will show in Section~\ref{main-sec} that the conclusions in 
Theorem~\ref{second.theorem}(a) can be made stronger under weaker assumptions. Complementing the condition
$\rho(Af')<q$ in Theorem~\ref{second.theorem}(b), some new conditions implying the conclusion $q=1$ will be discovered.
Examples on higher order duality as well as on the cases where a solution is dual to more than one
coefficient will be discussed in Sections~\ref{higher-order} and \ref{multiple-duality}, respectively.
Two open problems are formulated in the hope that these findings would give raise to further discussions
in the future.


\section{Preliminaries on exponential polynomials}\label{preliminaries.sec}


We need to introduce several concepts some of which are new.

\begin{definition}\label{commensurable}
(\cite[p.~214]{Langer}) Let $f$ in \eqref{normalized-f} be a simple exponential polynomial.
If there exists a constant $w\in\C\setminus\{0\}$ such that $w_j/w$ is a positive integer for every
$j=1,2,\ldots,m$, then the (non-zero) frequencies of $f$ are said to be \emph{commensurable},
and $w$ is called a \emph{common factor}.
\end{definition}

For example, $f(z)=e^{\pi z}+3e^{2\pi z}+ze^{3\pi z}$ and $g(z)=e^{4iz}+e^{6iz}$ are simple exponential polynomials, both of their frequencies are commensurable, and examples for common factors are $\pi,\pi/2$ for $f$ and $i,2i$ for $g$. In particular, a common factor is not unique.

Note that it is usual to say that non-zero real numbers $a$ and $b$ are commensurable if their ratio $a/b$ is a rational number. Equivalently, there exist a real number $c$ and integers $m$ and $n$ such that $a=mc$ and $b=nc$. 
In Definition~\ref{commensurable} we are concerned with a simple exponential polynomial and a fixed $\theta\in[0, 2\pi)$, and thus all the non-zero frequencies are of the form $w_j=r_je^{i\theta}$ for $r_j>0$, and the ratio of $w_j$ and $w_i$ is 
	$$
	\frac{w_j}{w_i}=\frac{r_j}{r_i}=\frac{w_j/w}{w_i/w}.
	$$ 
This is a positive rational number for a common factor $w$. If we consider the dual exponential polynomial as well, 
all those frequencies are commensurable in the usual sense. 

If $f$ is a simple exponential polynomial of order one with constant multipliers, and if its
frequencies are commensurable as in Frei's case (\ref{Frei-f}), then by the fundamental theorem of algebra, $f$ can be written as
    $$
    f(z)=A\prod_{j=1}^m\left(e^{wz}-\alpha_j\right),
    $$
where $A\neq 0$, $\alpha_j$'s are complex constants and $m$ is a positive integer. In particular, all the zeros of $f$ lie on at most $m$ lines.

We note that if the non-zero frequencies of $f$ are commensurable, then they are clearly linearly dependent
over rationals (see \cite{Moreno} for results in this direction), but not the other way around.
For example, the points $w_1=1, w_2=\sqrt{2}, w_3=\sqrt{2}-1$ are linearly dependent over rationals but
not commensurable.

\begin{definition}\label{s-commensurable}
Suppose that $f$ and $g$ are dual exponential polynomials with commensurable frequencies
$\{w_j\}$ $(j>0)$ and $\{\lambda_i\}$ $(i>0)$, respectively, sharing the same common factor $w$ but with opposite signs.  If the points $w_j+\lambda_i$ 
are on one ray including the origin for all $i,j>0$, then
$f$ and $g$ are called \emph{strongly dual exponential polynomials}.
\end{definition}

For example, the functions $f(z)=1+ze^z+2e^{3z}$ and $g(z)=1-e^{-z}$ are strongly dual exponential polynomials,
while $f(z)$ and $h(z)=g(z)+2z^2e^{-2z}$ are not.
Note that if $\arg(w_j)=\theta$, then $\arg(\lambda_j)=\theta+\pi$ by duality, and moreover, if $w_j+\lambda_i\neq 0$,
then precisely one of $\arg(w_j+\lambda_i)=\theta$ or $\arg(w_j+\lambda_i)=\theta+\pi$ holds for all $i,j>0$.
Alternatively, strong duality of $f$ and $g$ of order $q$ can be expressed as follows: There exists a 
non-zero constant $w$  such that 
	\begin{equation}\label{fg}
	f(z) = \sum_{j = 0}^m F_j(z)(e^{wz^q})^j\quad\textnormal{and}\quad  g(z) = \sum_{i = 0}^m G_i(z)(e^{-wz^q})^i,
	\end{equation}
where $F_j, G_i$ are exponential polynomials of order $\leq q-1$. Hence $f$ is a polynomial in $e^{wz^q}$ and 
$g$ is a polynomial in $e^{-wz^q}$, with smaller exponential polynomials as multipliers. Using the notation above,  
$f\in \operatorname{Exp}_{q-1}[e^{wz^q}]$ and $g\in \operatorname{Exp}_{q-1}[e^{-wz^q}]$. 
Differing from the situation in \eqref{normalized-f}, some of the multipliers $F_j, G_i$ $(i, j>0)$ in \eqref{fg} must suitably vanish identically so that only one of $j-i\geq 0$ or $j-i\leq 0$ always holds for all non-vanishing multipliers $F_j, G_i$ $(i, j>0)$. 
This is a consequence of Definition~\ref{s-commensurable}.  

We may think that being strong in our duality means that the product of $f(z)-F_0(z)$ and $g(z)-G_0(z)$ becomes again a commensurable exponential polynomial with either $w$ or $-w$ as a common factor. In the case when both $F_0(z)$ and $G_0(z)$ are constant, each product of the derivatives $f^{(k)}(z)$ and $g^{(\ell)}(z)$, $k, \ell\in\mathbb{N}$, is a commensurable exponential polynomial with the same common factor as the product of $f(z)-F_0(z)$ and $g(z)-G_0(z)$.

\begin{definition}
(\cite{WGH}) Denote the set of complex conjugate frequencies of the function $f$ in \eqref{normalized-f} by
$W_f=\{\overline{w}_0,\overline{w}_1,\ldots,\overline{w}_m\}$, where
$\overline{w}_0=0$ is related to the multiplier $F_0(z)\not\equiv 0$, and $W_f=\{\overline{w}_1,\ldots,\overline{w}_m\}$ when $F_0(z)\equiv 0$. Denote the convex hull of the set $W_f$ by $\co (W_f)$, and  let $C(\co(W_f))$ denote the circumference of $\co(W_f)$.
\end{definition}

The set $\co(W_f)$ is defined as the intersection of all closed convex sets containing $W_f$, and as
such it is either a convex polygon or a line segment. 
The latter occurs when $f$ is simple, and, in particular, when $w_1, \ldots, w_m$ are commensurable. 
The vertices of $\co (W_f)$ are formed by some (possibly all) of the points
$\overline{w}_0,\overline{w}_1,\ldots,\overline{w}_m$. The circumference
$C(\co(W_f))$ of $\co(W_f)$ plays an important role in
describing the value distribution of $f$, see \cite{division, WH, Stein1}.

Let $h$ be a quotient of two transcendental exponential polynomials, say
	$$
	h(z)=f(z)/g(z),
	$$
where $f$ is of the form \eqref{normalized-f} and $g$ is an exponential polynomial
of the normalized form
	$$
	g(z)=G_0(z)+G_1(z)e^{w_1z^q}+\cdots+G_m(z)e^{w_mz^q}.
	$$
In these representations of $f$ and $g$ for the quotient $h$, we allow that
some of the multipliers $F_j$ or $G_j$ may vanish identically, but we suppose that
the matching multipliers $F_j$ and $G_j$ do not both vanish
identically for any $j$.

For the quotient $h=f/g$, define the set $W_h=\{\overline{w}_0,\overline{w}_1,\ldots,\overline{w}_m\}$. 
The proximity function of $h$ is
	\begin{equation}\label{m.eq}
	m(r,h)=\big(C(\co(W_h))-C(\co(W_g))\big)\frac{r^q}{2\pi}+o(r^q),
	\end{equation}
see \cite[Satz 1]{Stein2}. In particular, if $g\equiv 1$, then $W_g=\{0\}$ and $C(\co(W_g))=0$.
This yields \cite[Satz 1]{Stein1} as a special case, namely
	\begin{equation}\label{T}
	T(r,f)=m(r,f)=C(\co(W^0_f))\frac{r^q}{2\pi}+o(r^q),
	\end{equation}
where $W^0_f=W_f\cup\{0\}$. The estimates \eqref{m.eq} and \eqref{T} are consistent with the estimate
	$$
	m\left(r,\frac{f'}{f}\right)=o(T(r,f)),
	$$
known as the lemma on the logarithmic derivative, since $W_{f'/f}=W_f$ holds for any given
exponential polynomial $f$ of the form \eqref{normalized-f}. We also point out that $W_{f/f'}=W_{f'}$.
This fact will be used in proving our main results in Section~\ref{main-sec}.

%
\section{The main results}\label{main-sec}
%

Motivated by Example~\ref{Kazuya-ex}, we improve Theorem~\ref{second.theorem}(a)
under weaker assumptions~on~$B(z)$.

\begin{theorem}\label{dual-thm}
Suppose that $f$ and $A(z)$ in \eqref{ldeAB} are transcendental exponential polynomials,
and that $B(z)$ is an entire function satisfying $T(r,B)=o(T(r,A))$. Then the following assertions~hold.
\begin{itemize}
\item[\textnormal{(a)}] 
$f$ and $A(z)$ are dual exponential polynomials of order $q\in\N$, $f$ has the normalized representation
\eqref{simple-f}, and $B(z)$ is an exponential polynomial of order
$\rho(B)\leq q-1$.
\item[\textnormal{(b)}]
The frequencies of $f$ are commensurable if and only
if the frequencies of $A(z)$ are commensurable. In both cases, $f$ and $A(z)$ are
strongly dual exponential \mbox{polynomials.}
\end{itemize}
\end{theorem}

\begin{proof}
(a) Suppose that $0\leq \rho(f)\leq \rho(A)-1$. The case $\rho(f)=0$ is not possible because $f$ is transcendental. 
Hence $\rho(f)\geq 1$. But now
	\begin{equation*}
	|A|\leq |f''/f'|+|B||f/f'|
	\end{equation*}
and the assumption $T(r,B)=o(T(r,A))$ imply
	$$
	T(r,A)=m(r,A)\leq m(r,B)+O\left(r^{\rho(A)-1}\right)=o(T(r,A)),
	$$
which is a contradiction. Here we have used \eqref{m.eq} for $h=f/f'$ and $g=f'$, as well as
\eqref{T} for $A$ in place of~$f$.
The following two cases are also impossible by the proof of \cite[Theorem~3.6]{GOP}:
	\begin{itemize}
	\item[(1)] $\rho(f)=\rho(A)$ and either $F_0(z)\equiv 0$ or $F'_0(z)\not\equiv 0$.
	\item[(2)] $\rho(f)\geq \rho(A)+1$.
	\end{itemize}
Thus $\rho(f)=\rho(A)=q\geq 1$ and $f$ has the representation \eqref{simple-f}. 

We proceed to prove that $f$ and $A(z)$ are dual exponential polynomials.
Using \eqref{ldeAB}, we find that
	$$
	m\left(r,\frac{Af'}{f}\right)=O(\log r)+m(r,B)=o(T(r,A))=o\left(r^q\right).
	$$
The formula (7.3) in \cite{WGH} should be replaced by this. Thus the formula (7.7) in \cite{WGH} holds, and the reasoning in \cite{WGH} shows that $f$ and $A(z)$ are dual exponential polynomials.

To complete the proof of (a), it suffices to prove that $B(z)$ is an exponential polynomial of order
$\rho(B)\leq q-1$.
Since the frequencies $w_j$ of $f$ are all on one ray, we may  appeal to a rotation, and suppose that $w_1,\ldots,w_m\in\R_+$. By
renaming the frequencies $w_j$, if necessary, we may further suppose
that $0<w_1<\cdots<w_m$. Thus the dual coefficient must
be of the form
	\begin{equation}\label{AA}
	A(z)=A_0(z)+\sum_{j=1}^k A_j(z)e^{-\lambda_j z^q},
	\end{equation}
where $A_j(z)\not\equiv 0$ for all $j\in \{1,\ldots,k\}$ and $\lambda_1,\ldots,\lambda_k\in\R_+$.
Renaming the frequencies $\lambda_j$, if necessary, we may suppose that $0<\lambda_1<\cdots <\lambda_k$.
Write
	$$
	f'(z)=\sum_{j=1}^m G_j(z)e^{w_jz^q}\quad\textnormal{and}\quad
	f''(z)=\sum_{j=1}^m H_j(z)e^{w_jz^q},
	$$
where $G_j(z)=F_j'(z)+qw_jz^{q-1}F_j(z)\not\equiv 0$ and $H_j(z)=G_j'(z)+qw_jz^{q-1}G_j(z)\not\equiv 0$. Next, write $-Af'=Bf+f''$
in the form
	\begin{equation}\label{substitution}
	-\left(\sum_{j=1}^kA_je^{-\lambda_jz^q}\right)\left(\sum_{j=1}^m G_je^{w_jz^q}\right)=cB+\sum_{j=1}^m (A_0G_j+BF_j+H_j)e^{w_jz^q}.
	\end{equation}
From \eqref{substitution} we find that $B$ is an exponential polynomial of order $\rho(B)\leq q$.
In fact, from \eqref{T} and the assumption $T(r,B)=o(T(r,A))$, it follows that $\rho(B)\leq q-1$.

(b) We begin with some preparations. From \cite{Stein1} and \cite{Stein2}, we have
    $$
    m\left(r,\sum_{j=1}^kA_je^{-\lambda_jz^q}\right)=T\left(r,\sum_{j=1}^kA_je^{-\lambda_jz^q}\right)
    =\frac{\lambda_k}{\pi}r^q+o(r^{q}),
    $$
and
    \begin{eqnarray*}
    &&m\left(r,\bigg\{cB+\sum_{j=1}^m (A_0G_j+BF_j+H_j)e^{w_jz^q}\bigg\}\Big/\sum_{j=1}^m G_je^{w_jz^q}\right)\\
    &&\qquad \qquad=
    \frac{2w_m-2(w_m-w_1)}{2\pi}r^q+o(r^{q})=
    \frac{w_1}{\pi}r^q+o(r^{q}).
    \end{eqnarray*}
Therefore, we deduce that
	\begin{equation}\label{ordered-leading-coefficients}
	0<\lambda_1<\cdots <\lambda_k=w_1<\cdots <w_m.
	\end{equation}
Thus from \eqref{substitution}, it follows that 
	\begin{equation}\label{B}
	-A_kG_1=cB.
	\end{equation}
If $A_0G_m+BF_m+H_m\not\equiv0$, then from \cite[Theorem 2.2]{WH} and \eqref{ordered-leading-coefficients}, we get
    \begin{eqnarray*}
    N(r,0,L)&=&\frac{2(w_m-w_1)+2(\lambda_k-\lambda_1)}{2\pi}r^q+O(r^{q-1}+\log r)\\
    &=&\frac{w_m-\lambda_1}{\pi}r^q+O(r^{q-1}+\log r),\\
    N(r,0,R)&=&\frac{w_m}{\pi}r^q+O(r^{q-1}+\log r),
    \end{eqnarray*}
where $N(r,0,L)$ and $N(r,0,R)$ are the counting functions of zeros of the exponential
polynomials on the left-hand side and on the right-hand side of \eqref{substitution}, respectively. 
This implies $w_m=w_m-\lambda_1$, which is impossible. Thus we have
    \begin{equation}\label{maxcoe.eq}
    A_0G_m+BF_m+H_m\equiv 0.
    \end{equation}
Now \eqref{substitution} reduces to
    \begin{equation}\label{reduce.eq}
    -\left(\sum_{j=1}^kA_je^{-\lambda_jz^q}\right)\left(\sum_{j=1}^m G_je^{w_jz^q}\right)
    =cB+\sum_{j=1}^{m-1} (A_0G_j+BF_j+H_j)e^{w_jz^q}.
    \end{equation}
From the Borel-Nevanlinna theorem, and from $A_iG_j\not\equiv 0$ for $j\in\{1,2,\ldots,m\}$ and $i\in\{1,2,\ldots,k\}$, it follows that there are only two possibilities:
\begin{itemize}
\item[(I)]
For some pairs $(j,i)$, where $j\in\{1,2,\ldots,m\}$ and
$i\in\{1,2,\ldots,k\}$, there exists $\ell\in\{0,1,\ldots,m-1\}$ such that
    \begin{equation}\label{I.eq}
    w_j-\lambda_i=w_\ell.
    \end{equation}
\item[(II)]
For some pairs $(j,i)$, where $j\in\{1,2,\ldots,m\}$ and
$i\in\{1,2,\ldots,k\}$, there exist $s\in \{1,2,\ldots,m\}\setminus\{j\}$ and
$t\in \{1,2,\ldots,k\}\setminus\{i\}$ such that
    \begin{equation}\label{II.eq}
    w_j-\lambda_i=w_s-\lambda_t.
    \end{equation}
\end{itemize}

After these preparations we proceed to prove that the frequencies of $f$ are commensurable if and only if the frequencies of $A(z)$ are commensurable. By appealing to \eqref{ordered-leading-coefficients} and to a change of variable as in Example~\ref{Kazuya-ex}, we may suppose that $w_1=\lambda_k\in\N$.
Thus we prove that $w_j\in\N$ for $j\in\{1,\ldots,m\}$ if and only if $\lambda_i\in\N$
for $i\in\{1,\ldots,k\}$.
\begin{itemize}
\item[(i)]
Suppose that $w_j\in\N$ for $j\in\{1,\ldots,m\}$. From \eqref{ordered-leading-coefficients}, we see that
$w_m-\lambda_1=\max_{j,i}\{w_j-\lambda_i\}$ and $w_m-\lambda_1>w_j-\lambda_i$ for any $j\neq m$ and $i\neq 1$. Hence, from \eqref{I.eq} and \eqref{II.eq}, there exists $p<m$ such that $w_m-\lambda_1=w_p$,
which implies that $\lambda_1\in\N$. In addition, from \eqref{ordered-leading-coefficients}, we have $w_m-\lambda_2>w_j-\lambda_i$ for any $j\neq m$ and $i>2$. Thus, from \eqref{I.eq} and \eqref{II.eq}, there are only two possibilities: (1) There exists $p<m$ such that $w_m-\lambda_2=w_p-\lambda_1$. (2) There exists $p<m$ such that $w_m-\lambda_2=w_p$. In both cases, it follows that $\lambda_2\in\N$. Repeating this argument for $k$ times gives us $\lambda_i\in\N$ for $i\in\{1,\ldots,k\}$.
\item[(ii)]
Suppose that $\lambda_i\in\N$ for $i\in\{1,\ldots,k\}$. 
From \eqref{ordered-leading-coefficients}, we have $\lambda_k=w_1$, and consequently $w_1\in\N$. Moreover, 
from \eqref{ordered-leading-coefficients}, we have $w_2-\lambda_k<w_j-\lambda_i$ for any $j>1$ and $i\neq k$. 
Thus, from \eqref{I.eq} and \eqref{II.eq}, there are only two possibilities: There exists $p<k$ such that either $w_2-\lambda_k=w_1-\lambda_p$ or $w_2-\lambda_k=w_1$. In both cases, we have $w_2\in\N$. Repeating this argument for $m$ times gives us $w_j\in\N$ for $i\in\{1,\ldots,m\}$.

\end{itemize}

If the frequencies are commensurable for one of $f,A(z)$, then they are commensurable for
both of $f,A(z)$ by the reasoning above. The remaining fact that $f$ and $A(z)$ are strongly dual
exponential polynomials now follows by \eqref{ordered-leading-coefficients}.
\end{proof}

The assumption $\rho(Af')<\rho(f)$ in Theorem~\ref{second.theorem}(b) seems to be the only known
sufficient condition for the conclusion $q=1$. However, in the case of Frei's result \eqref{Frei-eqn}, we have
	$$
	A(z)f'(z)=e^{-z}\sum_{j=1}^m jC_je^{jz}=\sum_{j=0}^{m-1}(j+1)C_{j+1}e^{jz},
	$$
and so $\rho(Af')=\rho(f)=1$. This shows that $q=1$ may happen even if $\rho(Af')=\rho(f)$.
Theorem~\ref{one-term} below shows that $f$ having only one large exponential term is also a sufficient condition
for $q=1$. In contrast, if $A(z)$ has only one large exponential term, then $f$ can have multiple
large exponential terms as in \eqref{Frei-eqn}.

\begin{theorem}\label{one-term}
Suppose that $f(z)=F_0(z)+F_1(z)e^{wz^q}$ is a solution of \eqref{ldeAB}, where $A(z)$ is an
exponential polynomial and $B(z)$ is an entire function satisfying $T(r,B)=o(T(r,A))$. Then $q=1$, and there
are constants $c,b\in\C\setminus\{0\}$ and a non-trivial polynomial $P(z)$ such that
	$$
	f(z)=c+be^{wz},\ A(z)=\frac{b}{c}P(z)-w+P(z)e^{-wz}\ \text{and}\ B(z)=-\frac{wb}{c}P(z).
	$$
\end{theorem}

\begin{proof}
We proceed similarly as in the proof of Theorem~\ref{dual-thm} until \eqref{reduce.eq}, which
now reduces to the form
	\begin{equation}\label{m=1}
    -\left(\sum_{j=1}^kA_je^{-\lambda_jz^q}\right)G_1e^{wz^q}=cB,
	\end{equation}
where $F_0(z)\equiv c\in\C\setminus\{0\}$.
Hence $k=1$, and consequently $A(z)$ reduces to the form
	$$
	A(z)=A_0(z)+A_1(z)e^{-wz^q}.
	$$
From \eqref{m=1} and \eqref{maxcoe.eq}, with $k=1=m$, we find that
	$$
	-A_1G_1=cB\quad\textnormal{and}\quad -A_0G_1=BF_1+H_1.
	$$
In other words,
	\begin{equation}\label{G-has-no-zeros}
	c^{-1}A_1G_1F_1=A_0G_1+H_1=A_0G_1+G_1'+qwz^{q-1}G_1.
	\end{equation}
Dividing both sides of \eqref{G-has-no-zeros} by $G_1$, we observe that at every zero of $G_1$
the right-hand side has a pole but the left-hand side does not. Thus $G_1$ has no zeros, and so
we may write it in the form $G_1=e^g$, where $g(z)=a_{q-1}z^{q-1}+\cdots+a_0$
is a polynomial of degree $\leq q-1$. Since
	$$
	G_1=F_1'+qwz^{q-1}F_1=e^g,
	$$
we obtain $\left(F_1(z)e^{wz^q}\right)'=e^{wz^q+g(z)}$, and consequently
	\begin{equation}\label{F1}
	F_1(z)e^{wz^q}=\int^z e^{w\zeta^q+a_{q-1}\zeta^{q-1}+\cdots+a_0}\, d\zeta.
	\end{equation}
Here the right-hand side is an exponential polynomial, which happens only if $q=1$.

Since $q=1$, we see from \eqref{F1} that $F_1(z)$ reduces to a non-zero constant, say $F_1(z)\equiv b$.
Thus $f(z)=c+be^{wz}$, and we have $G_1(z)\equiv w b$ and $H_1(z)\equiv w^2b$. A substitution to
\eqref{G-has-no-zeros} followed by a simplification gives
	$$
	\frac{b}{c}A_1=A_0+w.
	$$
There is no restriction for $A_1$ other than the fact that $A$ is an exponential polynomial. Thus
we may suppose that $A_1$ is any non-trivial polynomial, say $A_1=P$. This gives us $A_0=\frac{b}{c}P-w$,
and finally $B=-\frac{wb}{c}P$.
\end{proof}

Example~\ref{Ozawa-ex} shows that the coefficient $B(z)$ in \eqref{ldeAB} can be a polynomial.
Next, we prove that this is equivalent to $A_0(z)$ in \eqref{AA} being a polynomial, and reveal another
sufficient condition for the conclusion $q=1$.

\begin{proposition}\label{polynomials-lemma}
Under the assumptions of Theorem~\ref{dual-thm}, the term $A_0(z)$ of $A(z)$ in \eqref{AA}
is a polynomial if and only if $B(z)$ is a polynomial. Moreover, if the multipliers of $f$ and
of $A(z)$ are constants, then $q=1$ and $B(z)$ is a constant function.
\end{proposition}

\begin{proof}
From the proof of Theorem~\ref{dual-thm} we find that \eqref{maxcoe.eq} holds, that is,
	\begin{equation}\label{subst}
	A_0G_m+BF_m+H_m=0,
	\end{equation}
where
	\begin{eqnarray*}
	G_m &=& F_m'+w_mqz^{q-1}F_m,\\
    H_m &=& F_m''+2w_mqz^{q-1}F_m'+\left(w_mq(q-1)z^{q-2}+w_m^2q^2z^{2q-2}\right)F_m.
	\end{eqnarray*}
Thus $F_m$ solves the second order differential equation
	\begin{equation}\label{Fm-equation}
	F''_m + P(z)F'_m+Q(z)F_m=0,
    \end{equation}
where
	\begin{eqnarray*}
	P(z) &=& 2w_mqz^{q-1}+A_0,\\
	Q(z) &=& w_mqz^{q-1}A_0+w_mq(q-1)z^{q-2}+w_m^2q^2z^{2(q-1)}+B.
	\end{eqnarray*}
	
Suppose first that $A_0(z)$ is a polynomial. If $B(z)$ is transcendental, then it follows from
\eqref{Fm-equation} and \cite[Corollary~1]{Gary2} that $\rho(F_m)=\infty$, which is a contradiction. Hence $B(z)$ must
be a polynomial. Conversely, suppose that $B(z)$ is a polynomial. Suppose on the contrary to the
assertion that $A_0(z)$ is a transcendental exponential polynomial. Then there exists an open sector $S$
such that $A_0(z)$ blows up exponentially in $S$. Using \cite[Corollary~1]{Gary3} and
$\rho(F_m)\leq q-1$ in \eqref{Fm-equation}, we obtain on almost every ray in $S$ that
	$$
	|w_mqz^{q-1}||A_0(z)|
	\leq O\left(|z|^{\max\{2q,\deg(B)\}}\right)+O\left(|z|^{q-2+\varepsilon}|A_0(z)|\right).
	$$
However, this is obviously a contradiction, and hence $A_0(z)$ is a polynomial.

Finally, suppose that the multipliers of $f$ and of $A(z)$ are constants. From \eqref{B}, we find that
$B(z)=Cz^{q-1}$ for some constant $C\in\C\setminus\{0\}$. Since $F_m(z)$ is a non-zero constant function, it follows
that the coefficient $Q(z)$ in \eqref{Fm-equation} vanishes identically. But this is not possible because $A_0(z)$
is a constant function, unless $q=1$.
\end{proof}

\begin{remark}
(a) The equation \eqref{Fm-equation} implies that every possible zero of $F_m$ is simple.

(b) Assuming that $A_0(z)$ is a polynomial, we give an alternative proof for the fact that $B(z)$ is a polynomial. 
We already know from Theorem~\ref{dual-thm} that $B(z)$ is of order $\leq q-1$.
Since the non-zero frequencies of $A(z)$ are
all on one ray by duality, it follows that the plane divides into $2q$ sectors of opening $\pi/q$ such
that in every other sector $A(z)$ either blows up exponentially or is asymptotic
to the polynomial $A_0(z)$. In the latter case, if $A_0(z)\equiv 0$, then $A(z)$ decays to
zero exponentially. Thus, from \cite[Theorem~7]{Gary2}, we deduce that $B(z)$ is a
polynomial. Note, in particular, that the constant $\mu$ in \cite[Theorem~7]{Gary2} satisfies
$\mu=\pi/q$. 
\end{remark}

\bigskip
\noindent
\textbf{Open problem 1.}
\emph{Under the assumptions of Theorem~\ref{dual-thm}, is it always true that $q=1$ and $B(z)$ is a polynomial?}

\bigskip

This problem is fragile in the sense that the desired conclusion is not valid if a minor modification in the assumptions of  Theorem~\ref{dual-thm} is performed. For example, the differential equation 
	$$
	f''-\bigl(qwz^{q-1} + z^{-1}e^{-wz^q}\bigr)f'-q(q-1)wz^{q-2}f=0
	$$  
possesses an exponential polynomial solution $f(z)=e^{wz^q}-1/(q-1)$ for any $q\geq2$. Moreover, the function $f(z)=e^{z^2}+1$ satisfies the differential equations 
	\begin{equation}\label{two-equations}
	\begin{split}
	f''+\left(\frac{e^{-z^2}-1}{2z}-2z\right)f'-f &= 0,\\
	f''-\frac{e^{-z^2}(z-1)+4z^2+z+1}{2z}f'+(z-1)f &= 0.
	\end{split}
	\end{equation}
The transcendental coefficients in \eqref{two-equations} are entire exponential polynomials with rational multipliers 
because $z=0$ is a removable singularity for both.


\section{Duality for higher order functions}\label{higher-order}

Next we construct examples of differential equations of order $n\geq 2$ having
an exponential polynomial solution $f$ of order $\rho(f)=n-1$ which is dual with one of the coefficients. 

\begin{example}\label{ex1}
If $H$ is an arbitrary entire function, then $f(z)=e^{z^2}+1$ solves
	\begin{equation}\label{third-order-ex}
	\begin{split}
	f'''+\left(1+e^{-z^2}\right)Hf''-(6+4z^2)f'-(2+4z^2)Hf&=0,\\
	f'''-2zf''+(H-4+He^{-z^2})f'-2zHf &= 0.
	\end{split}	
	\end{equation}
A particularly interesting case is when $H$ is either a polynomial or an exponential polynomial of order one. 
Thus either of the two possible coefficients can be dual with $f$. Examples of second order
dual solutions for third order equations can be found in \cite{WGH} but for
polynomial coefficients only.

We can use the relation  $zf''(z)=(2z^2+1)f'$ to see that, in addition to \eqref{third-order-ex}, the function 
$f(z)=e^{z^2}+1$ satisfies the equations 
	\begin{eqnarray*}
	f'''(z)-2zf'(z)-4f'(z)&=&0,\\
	(1+e^{-z^2})f'(z)-2zf(z)&=&0,\\
	(1+e^{-z^2})f''(z)-(2+4z^2)f(z)&=&0.
	\end{eqnarray*}
\end{example}

\begin{example}\label{ex2}
If $H$ is an arbitrary entire function, then $f(z)=e^{z^3}+1$ solves
	\begin{eqnarray*}
	f^{(4)}+\left(1+e^{-z^3}\right)Hf'''-9z^4f''-30\left(2+3z^3\right)f'-\left(6+54z^3+27z^6\right)Hf &=& 0,\\
	f^{(4)}-3z^2f'''+\left(H-18z+He^{-z^3}\right)f''-18f'-H\left(6z+9z^4\right)f &=& 0,\\
	f^{(4)}-3z^2f'''-27zf''+\left(H+27z^3+He^{-z^3}\right)f'-3z^2Hf &=& 0.
	\end{eqnarray*}
A particularly interesting case is when $H$ is an exponential polynomial of order at most two. 
Thus all three of the possible coefficients can be dual with $f$. Previous examples of third order dual solutions
do not seem to be known.

As in the previous example, we can use the relations $zf''(z)=(3z^3+2)f'(z)$ and $zf'''(z)=(3z^3+1)f''(z)+9z^2f'(z)$
to see that $f(z)=e^{z^3}+1$ satisfies the equations
	\begin{eqnarray*}
	f^{(4)}(z)-3z^2f'''(z)-18zf''(z)-18f'(z)&=&0,\\
	(1+e^{-z^3})f'(z)-3z^2f(z)&=&0,\\
	(1+e^{-z^3})f''(z)-(6z+9z^4)f(z)&=&0,\\
	(1+e^{-z^3})f'''(z)-(6+54z^3+27z^6)f(z)&=&0.
	\end{eqnarray*}
\end{example}

In light of Open problem 1 and the examples just discussed, it is natural to pose our second open problem.

\bigskip
\noindent
\textbf{Open problem 2.}
\emph{If a solution and the dominant coefficient are dual exponential polynomials of order $q$, then is 
the differential equation in question of order at least $q+1$?}

\bigskip
For the fragility of this problem, recall the equations \eqref{two-equations} satisfied by $f(z)=e^{z^2}+1$.
Moreover, the function $f(z)=e^{z^3}+1$ satisfies the third order equation
	\begin{equation*}
	f'''+\left(\frac{e^{-z^3}-1}{2z}\right)f''-3z(3z^3+5)f'-\frac{3}{2}(3z^3+2)f=0
	\end{equation*}
with entire coefficients. 

As the first initial step to knowing more about Open problem 2, we make a summary of the fundamental ideas 
in constructing Examples~\ref{ex1} and \ref{ex2}.

\begin{lemma}\label{Tohge-lemma} 
The function $f(z)=e^{z^q}+1$, $q\in\mathbb{N}$, possesses the following two properties:
\begin{itemize}
\item[ \textnormal{(i)}] $\displaystyle
(1+e^{-z^q})f^{(j+1)}(z)=\sum_{k=0}^{j} P_{j,k}(z)f^{(k)}(z), \quad j\in\mathbb{N}\cup\{0\}$,
\item[\textnormal{(ii)}]
$\displaystyle
f^{(q+1)}(z)=\sum_{\ell=1}^{q} Q_{\ell}(z)f^{(\ell)}(z)$,
\end{itemize}
where the $P_{j,k}(z)$ and $Q_{\ell}(z)$ are non-zero polynomials satisfying
\begin{itemize}
\item[\textnormal{(a)}]
$\displaystyle
\left\{
\begin{array}{l}
P_{j+1,j+1}(z)=P_{j,j}(z)-qz^{q-1}, \quad P_{0,0}(z)=qz^{q-1}, \\
P_{j+1,k}(z)=P_{j,k}'(z)+qz^{q-1}P_{j,k}(z)+P_{j, k-1}(z), \quad P_{j,-1}(z)\equiv 0, \quad 1\leq k\leq j, \\
P_{j+1,0}(z)=P_{j,0}'(z)+qz^{q-1}P_{j,0}(z),
\end{array}
\right.
$
\item[\textnormal{(b)}]
$ \displaystyle
Q_{\ell}(z)=-\binom{q}{\ell-1}(e^{-z^q})^{(q-\ell+1)}e^{z^q}.
$
\end{itemize}
\end{lemma}

\begin{proof}
First, let us prove (i) by induction on $j$.
Of course, by taking their logarithmic derivatives, we have $(1+e^{-z^q})f'(z)=qz^{q-1}f(z)$ immediately, that is, the case when $j=0$ follows with $P_{0,0}(z)=qz^{q-1}$.
Assume (i) is true for each $j=0, 1, \ldots , n$. Then
\begin{eqnarray*}
(1+e^{-z^q})f^{(n+2)}(z)
&=& qz^{q-1}e^{-z^q}f^{(n+1)}+\sum_{k=0}^n \bigl\{ P_{n,k}'(z)f^{(k)}(z)+P_{n,k}(z)f^{(k+1)}(z)\bigr\}\\
&=& qz^{q-1}(1+e^{-z^q})f^{(n+1)}(z)+\bigl\{P_{n,n}(z)-qz^{q-1}\bigr\}f^{(n+1)}(z)+\\
& & +\sum_{k=1}^n \bigl\{ P_{n,k}'(z)f^{(k)}(z)+P_{n,k-1}(z)\bigr\} f^{(k)}(z)+P_{n,0}'(z)f(z) \\
&=& qz^{q-1} \sum_{k=0}^n P_{n,k}(z)f^{(k)}(z) + \bigl\{P_{n,n}(z)-qz^{q-1}\bigr\}f^{(n+1)}(z)+\\
& & +\sum_{k=1}^n \bigl\{ P_{n,k}'(z)f^{(k)}(z)+P_{n,k-1}(z)\bigr\} f^{(k)}(z)+P_{n,0}'(z)f(z) \\
&=& \bigl\{P_{n,n}(z)-qz^{q-1}\bigr\}f^{(n+1)}(z) + \\
& & +\sum_{k=1}^n \bigl\{ P_{n,k}'(z)f^{(k)}(z)+qz^{q-1}P_{n,k}(z)+P_{n,k-1}(z)\bigr\}f^{(k)}(z)+\\
& & +\bigl\{P_{n,0}'(z)+qz^{q-1}P_{n,0}(z)\bigr\}f(z),
\end{eqnarray*}
which is the one to be proved.

Second, let us calculate the $q$-th order derivative of the product $f'(z)e^{-z^q}=qz^{q-1}$.
The Leibniz rule gives
$$
\sum_{\ell=0}^q \binom{q}{\ell} f^{(\ell+1)}(z)(e^{-z^q})^{(q-\ell)}\equiv 0.
$$
Denoting $Q_{\ell+1}(z)=-\binom{q}{\ell}(e^{-z^q})^{(q-\ell)}e^{z^q}$ for $0\leq \ell \leq q-1$, we have
$$
f^{(q+1)}(z)=\sum_{\ell=0}^{q-1}Q_{\ell+1}(z)f^{(\ell+1)}(z)=\sum_{\ell=1}^qQ_{\ell}(z)f^{(\ell)}(z),
$$
as desired.
\end{proof}

\begin{example}
We may apply the two identities in Lemma~\ref{Tohge-lemma} to construct differential equations of arbitrary order.
Given any entire function $H$, we have the identity
$$
f^{(q+1)}(z)-\sum_{\ell=1}^{q} Q_{\ell}(z)f^{(\ell)}(z)=H(z)\left( (1+e^{-z^q})f^{(j)}(z)-\sum_{k=0}^{j-1} P_{j-1,k}(z)f^{(k)}(z) \right),
$$
that is, $f(z)=e^{z^q}+1$ solves
\begin{eqnarray*}
f^{(q+1)}(z)
&-&\sum_{\ell=j+1}^{q} Q_{\ell}(z)f^{(\ell)}(z)- \left( (1+e^{-z^q})H(z) +Q_j(z) \right) f^{(j)}(z)\\
&+&\sum_{\ell=1}^{j-1} \bigl(P_{j-1,\ell}(z) H(z)-Q_{\ell}(z)\bigr)f^{(\ell)}(z) +P_{j-1,0}(z)H(z)f(z)=0,
\end{eqnarray*}
where $1\leq j \leq q$, the sum $\sum_{\ell=1}^{j-1}$ is empty if $j=1$ and the sum $\sum_{\ell=j+1}^{q}$ is empty if $j=q$.
\end{example}


\section{Multiple duality}\label{multiple-duality}

The possibility that a solution $f$ would be dual to more than one
coefficient has not been studied rigorously. In this case there would
be at least two equally strong dominant coefficients, or, in the case of \eqref{ldeAB}, 
both coefficients $A(z),B(z)$ would be equally strong. For example, $f(z)=e^{-z}$ solves
	$$
	f''+e^zf'+(e^z-1)f=0
	$$
and is dual to both coefficients. Obviously the coefficients are not
dual to each other. More examples can be produced from Example~\ref{Ozawa-ex}.
Note that $f(z)=e^z$ solves \eqref{ldeAB} if $A(z)=-B(z)-1$. Hence $f$
is not necessarily dual with either of $A(z),B(z)$.

If $H$ is any entire function, then $f(z)=e^{z^q}$ solves
	$$
	f''+\left(H(z)-qz^{q-1}\right)f'-\left(q(q-1)z^{q-2}+qz^{q-1}H(z)\right)f=0.
	$$
This example is from \cite{Gary2}. Note that  $f(z)=e^{z^q}$ satisfies both $f'(z)-qz^{q-1}f(z)=0$ and $f''(z)-qz^{q-1}f'(z)-q(q-1)z^{q-2}f(z)=0$.

Recall \cite[Theorem~2.1]{GOP}, according to
which there cannot be even one ray on which $B(z)$ would be stronger
than $A(z)$ in the sense of the Phragm\'en-Lindel\"of indicator, for otherwise all solutions of 
\eqref{ldeAB} are of infinite order.
This happens, for example, when $A(z)$ and $B(z)$ are dual to each other.

\begin{example}
One may observe the necessity of the assumption on the {\it duality} of $f$ and $A(z)$ as well as that on the {\it dominance} of $A(z)$ over $B(z)$ by the following example: The function $f(z)=\bigl(e^z+e^{-z}\bigr)e^{z^q}$, $q\in\mathbb{N}$, satisfies
	\begin{eqnarray*}
	f''&+&\bigl\{H(z)e^z+H(z)e^{-z}-2qz^{q-1}\bigr\}f' \\
	&-&\bigl\{(qz^{q-1}+1)H(z)e^z+(qz^{q-1}-1)H(z)e^{-z}-q^2z^{2(q-1)}+q(q-1)z^{q-2}+1 \bigr\}f=0
	\end{eqnarray*}
for any entire function $H$.
When $q=1$, this becomes
$$
f''+\bigl\{H(z)e^z+H(z)e^{-z}-2\bigr\}f'
-2H(z)e^zf=0
$$
with $f(z)=e^{2z}+1$.
Thus we may use it in order to observe the duality of $A(z)$ and $B(z)$ by several choices of $H$ such as $H(z)=e^{nz}$ for $n\in\mathbb{Z}$ or $H(z)=e^{iz}$.

Here we note that $f(z)=F(z)e^{z^q}$ satisfies $\dfrac{f'}{f}=\dfrac{F'}{F}+qz^{q-1}$ and
$$
\frac{f''}{f}=\frac{F''}{F}+2qz^{q-1}\frac{f'}{f}+\bigl(q(q-1)z^{q-2}-q^2z^{2(q-1)}\bigr)
$$
so that there is no large freedom to choose the function $F$.
For example, taking an Airy function as $F$, we cannot have our desired equation $f''+A(z)f'+B(z)f=0$ with the exponential polynomial coefficients $A(z)$ and $B(z)$.
\end{example}

\noindent
\textbf{Acknowledgements.} 
Ishizaki was supported by JSPS KAKENHI Grant Number 20K03658.
Wen was supported by the National Natural Science Foundation of China (No.~11971288 and No.~11771090) and Shantou University SRFT (NTF18029).

\medskip
\noindent
\emph{J.~Heittokangas}\\
\textsc{University of Eastern Finland, Department of Physics and Mathematics,
P.O.~Box 111, 80101 Joensuu, Finland}\\
\texttt{email:janne.heittokangas@uef.fi}

\medskip
\noindent
\emph{K.~Ishizaki}\\
\textsc{The Open University of Japan, Faculty of Liberal Arts, Mihama-ku, Chiba, Japan}\\
\texttt{email:ishizaki@ouj.ac.jp}

\medskip
\noindent
\emph{K.~Tohge}\\
\textsc{Kanazawa University, College of Science and Engineering, Kakuma-machi, Kanazawa 920-1192, Japan}\\
\texttt{email:tohge@se.kanazawa-u.ac.jp}

\medskip
\noindent
\emph{Z.-T.~Wen}\\
\textsc{Shantou University, Department of Mathematics, Daxue Road No. 243, Shantou 515063, China}\\
\texttt{e-mail:zhtwen@stu.edu.cn}

\end{document}